\documentclass[12pt,leqno,final,draft]{article}
\usepackage{amsfonts}
\usepackage{amsmath, dsfont, amsthm, amsfonts, amssymb, color}
\usepackage{mathrsfs}
\usepackage{color}
\usepackage{stmaryrd}
\usepackage[notcite,notref]{showkeys}
\makeatletter
\newcommand{\Spvek}[2][r]{%
  \gdef\@VORNE{1}
  \left(\hskip-\arraycolsep%
    \begin{array}{#1}\vekSp@lten{#2}\end{array}%
  \hskip-\arraycolsep\right)}

\def\vekSp@lten#1{\xvekSp@lten#1;vekL@stLine;}
\def\vekL@stLine{vekL@stLine}
\def\xvekSp@lten#1;{\def\temp{#1}%
  \ifx\temp\vekL@stLine
  \else
    \ifnum\@VORNE=1\gdef\@VORNE{0}
    \else\@arraycr\fi%
    #1%
    \expandafter\xvekSp@lten
  \fi}
\makeatother

\newtheorem{thm}{Theorem}[section]

\newtheorem{lem}[thm]{Lemma}
\newtheorem{prp}[thm]{Proposition}

\newtheorem{rem}[thm]{Remark}
\newtheorem{defn}[thm]{Definition}
\newcommand{\scr}[1]{\mathscr #1}
\definecolor{wco}{rgb}{0.5,0.2,0.3}

\numberwithin{equation}{section} \theoremstyle{remark}

\newcommand{\ua}{\uparrow}

\title{{\bf    Harnack Inequalities for McKean-Vlasov
SDEs Driven by Subordinate Brownian Motions}\footnote{Supported in
 part by  NNSFC (11801406, 11831015).}
}
\author{
{\bf     Chang-Song Deng $^{a)}$, Xing Huang $^{b)}$,    }\\
\footnotesize{  a)School of Mathematics and Statistics, Wuhan University, Wuhan 430072, China}\\
\footnotesize{ dengcs@whu.edu.cn }\\
\footnotesize{  b)Center for Applied Mathematics, Tianjin University, Tianjin 300072, China}\\
\footnotesize{  xinghuang@tju.edu.cn}}
\begin{document}
\allowdisplaybreaks
\def\R{\mathbb R}  \def\ff{\frac} \def\ss{\sqrt} \def\B{\mathbf
B} \def\W{\mathbb W}
\def\N{\mathbb N} \def\kk{\kappa} \def\m{{\bf m}}
\def\ee{\varepsilon}\def\ddd{D^*}
\def\dd{\delta} \def\DD{\Delta} \def\vv{\varepsilon} \def\rr{\rho}
\def\<{\langle} \def\>{\rangle} \def\GG{\Gamma} \def\gg{\gamma}
  \def\nn{\nabla} \def\pp{\partial} \def\E{\mathbb E}
\def\d{\text{\rm{d}}} \def\bb{\beta} \def\aa{\alpha} \def\D{\scr D}
  \def\si{\sigma} \def\ess{\text{\rm{ess}}}
\def\beg{\begin} \def\beq{\begin{equation}}  \def\F{\scr F}
\def\Ric{\text{\rm{Ric}}} \def\Hess{\text{\rm{Hess}}}
\def\e{\text{\rm{e}}} \def\ua{\underline a} \def\OO{\Omega}  \def\oo{\omega}
 \def\tt{\tilde} \def\Ric{\text{\rm{Ric}}}
\def\cut{\text{\rm{cut}}} \def\P{\mathbb P} \def\ifn{I_n(f^{\bigotimes n})}
\def\C{\scr C}      \def\aaa{\mathbf{r}}     \def\r{r}
\def\gap{\text{\rm{gap}}} \def\prr{\pi_{{\bf m},\varrho}}  \def\r{\mathbf r}
\def\Z{\mathbb Z} \def\vrr{\varrho}
\def\L{\scr L}\def\Tt{\tt} \def\TT{\tt}\def\II{\mathbb I}
\def\i{{\rm in}}\def\Sect{{\rm Sect}}  \def\H{\mathbb H}
\def\M{\scr M}\def\Q{\mathbb Q} \def\texto{\text{o}}
\def\Rank{{\rm Rank}} \def\B{\scr B} \def\i{{\rm i}} \def\HR{\hat{\R}^d}
\def\to{\rightarrow}\def\l{\ell}\def\iint{\int}
\def\EE{\scr E}\def\Cut{{\rm Cut}}
\def\A{\scr A} \def\Lip{{\rm Lip}}
\def\BB{\scr B}\def\Ent{{\rm Ent}}\def\L{\scr L}
\def\R{\mathbb R}  \def\ff{\frac} \def\ss{\sqrt} \def\B{\mathbf
B}
\def\N{\mathbb N}
\def\CC{\mathbb C}
\def\kk{\kappa} \def\m{{\bf m}}
\def\dd{\delta} \def\DD{\Delta} \def\vv{\varepsilon} \def\rr{\rho}
\def\<{\langle} \def\>{\rangle} \def\GG{\Gamma} \def\gg{\gamma}
  \def\nn{\nabla} \def\pp{\partial} \def\E{\mathbb E}
\def\d{\text{\rm{d}}} \def\bb{\beta} \def\aa{\alpha} \def\D{\scr D}
  \def\si{\sigma} \def\ess{\text{\rm{ess}}}
\def\beg{\begin} \def\beq{\begin{equation}}  \def\F{\scr F}
\def\Ric{\text{\rm{Ric}}} \def\Hess{\text{\rm{Hess}}}
\def\e{\text{\rm{e}}} \def\ua{\underline a} \def\OO{\Omega}  \def\oo{\omega}
 \def\tt{\tilde} \def\Ric{\text{\rm{Ric}}}
\def\cut{\text{\rm{cut}}} \def\P{\mathbb P} \def\ifn{I_n(f^{\bigotimes n})}
\def\C{\scr C}      \def\aaa{\mathbf{r}}     \def\r{r}
\def\gap{\text{\rm{gap}}} \def\prr{\pi_{{\bf m},\varrho}}  \def\r{\mathbf r}
\def\Z{\mathbb Z} \def\vrr{\varrho}
\def\L{\scr L}\def\Tt{\tt} \def\TT{\tt}\def\II{\mathbb I}
\def\i{{\rm in}}\def\Sect{{\rm Sect}}  \def\H{\mathbb H}
\def\M{\scr M}\def\Q{\mathbb Q} \def\texto{\text{o}} \def\LL{\Lambda}
\def\Rank{{\rm Rank}} \def\B{\scr B} \def\i{{\rm i}} \def\HR{\hat{\R}^d}
\def\to{\rightarrow}\def\l{\ell}
\def\8{\infty}\newcommand\I{\mathds 1}\def\U{\scr U}
\maketitle

\begin{abstract} The existence and uniqueness are established for McKean-Vlasov SDEs driven by L\'{e}vy
processes. By using an approximation
technique and coupling by change of measures, Harnack
inequalities are investigated for McKean-Vlasov SDEs
driven by subordinate Brownian motions.
\end{abstract} \noindent

 \vskip 2cm

\section{Introduction}

It is well known that solution to the
linear Fokker-Planck-Kolmogorov equation (FPKE)
(cf. \cite{BKRS}) can be constructed by the time marginal
distributions of solution to It\^{o} (distribution
independent) stochastic differential
equation (SDE), see e.g. \cite{IW}.
This means that we can describe FPKEs
by using a probabilistic approach (\cite{BR1, BR2, RXZ20}). However, many important partial differential equations (PDEs) for probability measures are nonlinear, see, for instance,
\cite{BKRS,CA, DV1, DV2, FG, Gu, V2} and
references therein. Such PDEs are also of Fokker--Planck type. Fortunately, nonlinear FPKEs are also closely connected
to the so-called
distribution dependent SDEs, also named
McKean-Vlasov SDEs in the literature, in which the coefficients
depend on the distribution of the solution.
Barbu and R\"{o}ckner \cite{BR1,BR2} investigated
one-to-one correspondence between nonlinear FPKEs with second-order differential operator and
McKean-Vlasov SDEs driven by Brownian motion,
see also \cite{HRW} for closely related results
on path dependent nonlinear FPKEs
and path-distribution dependent SDEs
with Brownian noise.

Recently, Jourdain, M\'{e}l\'{e}ard
and Woyczynski
in \cite{JMW} investigated McKean-Vlasov model
with multiplicative L\'{e}vy noises. For McKean-Vlasov
SDEs driven by additive L\'{e}vy processes,
Y. Song \cite{Song} applied Malliavin calculus
to get  exponential
ergodicity  in the total variance distance, while
Liang, Majka and Wang used a different
approach in \cite{LMW} to derive
exponential ergodicity in
the $L^1$-Wasserstein distance.

Let $\scr P$ be the family of all probability
measures on $\R^d$ equipped with the weak topology, and $\L_\zeta$ denote
the distribution of a random
variable $\zeta$. When a different probability
measure $\tilde{\P}$ is concerned, we use
$\L_\zeta|_{\tilde{\P}}$ to denote
the law of $\zeta$ under $\tilde{\P}$.
In this paper, we
consider the following McKean-Vlasov
SDEs driven by L\'evy processes:
\beq\label{E1}
\d X_t=b(t,X_t,\L_{X_t})\,\d t+\sigma(t)\,\d Z_t,
\end{equation}
where $b: [0,\infty)\times \R^d\times\scr P\to \R^d$ and
$\si:[0,\infty)\to \R^d\otimes\R^d$
are measurable and locally bounded, and $Z=\{Z_t\}_{t\geq 0}$
is a $d$-dimensional L\'{e}vy process with $Z_0=0$.

Note that $Z$ has stationary and independent
increments and almost surely c\`adl\`ag (right-continuous with finite left limits) paths $t\mapsto Z_t$. Since $Z$ is a (strong) Markov process, it is completely characterized
by the law of $Z_t$, hence by the characteristic function of $Z_t$. It is well known that
$$
    \E\e^{\i\<\xi,Z_t\>}=\e^{-t\psi(\xi)},
    \quad t>0,\;\xi\in\R^d,
$$
where the symbol (characteristic exponent)
$\psi:\R^d\to\CC$ is given by
the L\'evy--Khintchine formula
$$
    \psi(\xi)
    =-\i\<l,\xi\> + \frac12\<\xi, Q\xi\> +\int_{\R^d\setminus\{0\}} \left(1-\e^{\i\<\xi, x\>} + \i\<\xi, x\> \I_{(0,1)}(|x|)\right)\,\nu_Z(\d x),
$$
where $l\in\R^d$ is the drift coefficient, $Q$ is a nonnegative semidefinite $d\times d$ matrix, and $\nu_Z$ is
the L\'evy measure
on $\R^d\setminus\{0\}$ satisfying $\int_{\R^d\setminus\{0\}}(1\wedge|x|^2)\,\nu_Z(\d x)<\infty$. The L\'evy triplet $(l,Q,\nu_Z)$ uniquely determines $\psi$, hence $Z$ and the infinitesimal generator
of $Z$ is of the form
\begin{equation}\label{levy}
    \mathscr{A}f
    = \<l,\nabla f\> +\frac12 \<\nabla, Q\nabla\> f
    +\int_{\R^d\setminus\{0\}} \left(f(x+\cdot)-f- \<x,\nabla f\> \I_{(0,1)}(|x|)\right)\,\nu_Z(\d x)
\end{equation}
for $f\in C_b^2(\R^d)$.

The first contribution of the present
paper is the existence
and uniqueness of the solution
to \eqref{E1}, see Theorem \ref{TEU} below.
To this end, we shall follow the iteration
argument used in \cite{ W16};
moreover, we need to bound the moment for solutions
to L\'{e}vy-driven (distribution
independent) SDEs with one-sided Lipschitz
continuous drift.

The dimension-free Harnack inequality, initialized
in \cite{Wan97}, has become an efficient tool in stochastic analysis, and it can be used to study
the strong Feller property, heat
kernel estimates, transportation-cost
inequalities, hyperboundedness, and many more; we refer to the
monograph by F.-Y. Wang
\cite[Subsection 1.4.1]{Wbook} for an
in-depth explanation of its applications.

To establish Harnack inequality
for McKean-Vlasov SDEs with jumps, we will restrict ourselves to
the special case $Z_t=W_{S_t}$, where $W=\{W_t\}_{t\geq 0}$
is a standard Brownian motion on $\R^d$,
and $S=\{S_t\}_{t\geq 0}$ is a subordinator
independent of $W$. Then the equation \eqref{E1} reduces to
\beq\label{sbb}
\d X_t=b(t,X_t,\L_{X_t})\,\d t+\sigma(t)\,\d W_{S_t}.
\end{equation}
We will adopt absolutely continuous path to approximate the path of $S$ as
in \cite{WW,ZXC,DH}, and, as it turns out, this will
be crucial for our study. As before
(see e.g. \cite{WW, DH}), a coupling argument
and the Girsanov theorem will also be used.

Recall that a subordinator $S=\{S_t\}_{t\geq 0}$ is a nondecreasing
L\'{e}vy process on $[0,\infty)$, and it is uniquely
determined by its Laplace transform which is of the form
$$
    \E \,\e^{-r S_t}=\e^{-t\phi(r)},\quad r>0,t\geq0.
$$
The characteristic (Laplace)
exponent $\phi:(0,\infty)\rightarrow(0,\infty)$ is
a Bernstein function, i.e. a $C^\infty$-function such that $\phi\geq 0$ and with alternating derivatives $(-1)^{n+1} \phi^{(n)}\geq0$, $n\in\N$. Every such $\phi$ has a unique L\'{e}vy--Khintchine representation
\begin{equation}\label{bern}
    \phi(r)
    = \varrho\, r+\int_{(0,\infty)}\left(1-\e^{-rx}\right)\,
    \nu_S(\d x), \quad r>0,
\end{equation}
where $\varrho\geq0$ is the drift parameter and $\nu_S$ is a L\'{e}vy measure,
that is, a Radon measure on $(0,\infty)$ satisfying
$\int_{(0,\infty)}
    \left(1\wedge x\right)\,\nu_S(\d x)<\infty$.
We use \cite{SSV12} as our standard reference for Bernstein functions and subordinators.

The (random) time-changed process $(W_{S_t})_{t\geq0}$ is a rotationally invariant L\'{e}vy process with symbol $\phi(|\cdot|^2/2)$
and is called a subordinate Brownian motion.
If $S$ is
an $\alpha$-stable subordinator with Bernstein function
$\phi(r)=r^\alpha$ ($0<\alpha<1$),
then $(W_{S_t})_{t\geq0}$ is the well-known
$2\alpha$-stable L\'{e}vy process with
discontinuous sample paths and its generator
is given by the fractional
Laplacian operator $-\frac12\,(-\Delta)^\alpha$.
By choosing different Bernstein
functions, we can construct many other time-changed
Brownian motions. Thus, subordinate Brownian motions
form a very large class of L\'{e}vy processes. Nonetheless, compared with general L\'{e}vy processes, subordinate Brownian motions are much more tractable.

The remaining part of the paper is organized as follows.
In Section 2, we investigate the strong/weak existence
and uniqueness of solutions to McKean-Vlasov SDEs
driven by L\'{e}vy processes. By using an approximation
technique and coupling by change of measures, the dimension-free Harnack inequalities are established in Section 5.
Finally, the appendix contains a result concerning
moments for L\'{e}vy-driven (distribution
independent) SDEs, which has been used
in Section 2.

\section{Existence and uniqueness for McKean-Vlasov
SDEs with L\'{e}vy noises}

For $p\in [1,\infty)$, let
$$\scr P_p := \left\{\mu\in \scr P\,: \,\, \mu(|\cdot|^p):=\int_{\R^d}|x|^p\,\mu(\d x)<\infty\right\}.$$
It is well known that
$\scr P_p$ is a Polish space under the Wasserstein distance
$$\W_p(\mu_1,\mu_2):= \inf_{\pi\in \scr C(\mu_1,\mu_2)} \bigg(\int_{\mathbb{R}^d\times\mathbb{R}^d} |x-y|^p \,\pi(\d x,\d y)\bigg)^{1/p},\quad \mu_1,\mu_2\in \scr P_{p},$$ where $\scr C(\mu_1,\mu_2)$ is the set of all couplings
for $\mu_1$ and $\mu_2$.  Moreover,   the topology induced by $\W_p$ on $\scr P_p$ coincides with the weak topology.

We make the following
assumptions on the L\'{e}vy
measure $\nu_Z$ of $Z$ and the coefficient $b$:
There exists some $\theta\geq 1$ such that
\beg{enumerate}
\item[\textbf{(H1)}] $\int_{|x|\geq1}
|x|^{\theta}\,\nu_Z(\d x)<\infty$;
\item[\textbf{(H2)}] (Continuity)  For every $t\ge 0$, $b(t,\cdot,\cdot)$ is continuous on $\mathbb{R}^d \times\scr P_\theta$;
 \item[\textbf{(H3)}] (Monotonicity) There exist locally bounded functions  $\kappa_1: [0,\infty)\to \R$ and $\kappa_2: [0,\infty)\to [0,\infty)$ such that
 \beg{align*} &2\langle  b(t,x,\mu)- b(t,y,\nu), x-y\rangle\\
 &\qquad\le \kappa_1(t) |x-y|^2+ \kappa_2(t)\W_\theta(\mu,\nu)|x-y|,\quad t\ge 0,\,x,y\in \mathbb{R}^d,\,\mu,\nu\in
\scr P_\theta;\end{align*}
\item[\textbf{(H4)}] (Growth)  
    There exists a locally bounded function  $\Theta: [0,\infty)\to[0,\infty)$ such that
$$|b(t,0,\mu)|\le \Theta (t) \big\{1+(\mu(|\cdot|^\theta))^{1/\theta}\big\},\quad t\ge 0,\, \mu\in \scr P_\theta.$$
\end{enumerate}

\begin{rem}
It is well known that \textbf{\upshape(H1)} is equivalent
    to $\E |Z_t|^{\theta}<\infty$ for
    some \textup{(}or, equivalently, all\textup{)}
    $t>0$, cf. \cite[Theorem 25.3]{Sato}.
\end{rem}

\begin{defn}\label{def1} A c\`{a}dl\`{a}g adapted process $(X_t)_{t\geq
0}$ on $\mathbb{R}^d$ is called
a \textup{(}strong\textup{)} solution of \eqref{E1}, if
$$
\E\int_0^t|b(s,X_s,\L_{X_s})|\,\d s<\infty, \ \ t\geq 0.
$$
and $\P$-a.s.
$$
X_t=X_0+\int_0^tb(s,X_s,\L_{X_s})\,\d s+\int_0^t\si(s)
\,\d Z_s, \quad t\ge0.
$$

\beg{enumerate}
\item[\textup{(1)}] We say that \eqref{E1}
has strong \textup{(}or pathwise\textup{)} existence
and uniqueness in $\scr P_{\theta}$,
if for any $\F_0$-measurable random variable $X_0$ with $\L_{X_0}\in\scr P_{\theta}$, the equation has a unique solution $(X_t)_{t\geq
0}$ satisfying $\E|X_t|^\theta<\infty$ for all $t>0$.

\item[\textup{(2)}] A couple $(\tilde{X}_t
, \tilde{Z}_t)_{t\geq 0}$ is called a weak solution to \eqref{E1}, if $\tilde{Z}=(\tilde{Z})_{t\geq0}$ is a
L\'{e}vy process having the same symbol as $Z$
with respect to a
complete filtered probability space
$(\tilde{\Omega}, \{\tilde{\F}_t\}_{t\geq 0}, \tilde{\P})$,
and $(\tilde{X}_t)_{t\geq 0}$ satisfies
$$
    \d \tilde{X}_t=b(t,\tilde{X}_t,
    \L_{\tilde{X}_t}|_{\tilde{\P}})\,\d t
    +\sigma(t)\,\d \tilde{Z}_t.
$$

\item[\textup{(3)}] \eqref{E1} is said to have
weak uniqueness in $\scr P_\theta$, if any two weak solutions of
the equation with common initial distribution in $\scr P_\theta$ are equal in law. Precisely, if $(\tilde{X}_t
, \tilde{Z}_t)_{t\geq 0}$ with respect to
$(\tilde{\Omega}, \{\tilde{\F}_t\}_{t\geq 0}, \tilde{\P})$
and $(\bar{X}_t
, \bar{Z}_t)_{t\geq 0}$ with respect to
$(\bar{\Omega}, \{\bar{\F}_t\}_{t\geq 0}, \bar{\P})$
are weak solutions of \eqref{E1}, then
$\L_{\tilde{X}_0}|_{\tilde{\P}}
=\L_{\bar{X}_0}|_{\bar{\P}}\in\scr P_\theta$ implies
$\L_{\tilde{X}_t}|_{\tilde{\P}}
=\L_{\bar{X}_t}|_{\bar{\P}}$ for all $t>0$.
\eqref{E1} is said to have strong/weak well-posedness in $\scr P_{\theta}$ if it has strong/weak existence and uniqueness in $\scr P_\theta$.
\end{enumerate}
\end{defn}
\begin{thm}\label{TEU} 
Assume \textbf{\upshape(H1)}-\textbf{\upshape(H4)}. Then the following assertions hold.
\begin{enumerate}
\item[\textup{(1)}] The equation \eqref{E1} has strong well-posedness
in $\scr P_{\theta}$.
\item[\textup{(2)}] The equation \eqref{E1} has weak well-posedness
in $\scr P_{\theta}$.
\end{enumerate}
\end{thm}

We will prove Theorem \ref{TEU} by the argument
used in \cite{W16}. For fixed  $s\ge 0$ and $\F_s$-measurable  $\mathbb{R}^d$-valued random variable $X_{s,s}$ with $\E|X_{s,s}|^\theta<\infty$, set
$$X^{(0)}_{s,t}=X_{s,s}, \ \ \mu_{s,t}^{(0)}=\L_{X^{(0)}_{s,t}},\quad t\ge s.$$
For $n\in\N$, let $(X_{s,t}^{(n)})_{t\ge s}$ solve
the classical (distribution independent) SDE
\beq\label{EN}
\d X^{(n)}_{s,t}= b(t,X^{(n)}_{s,t}, \mu_{s,t}^{(n-1)})\,\d t + \si(t)\,\d Z_t,\quad  t\ge s,
\end{equation}
with $X_{s,s}^{(n)}=X_{s,s}$, where $\mu_{s,t}^{(n-1)}:=\L_{X_{s,t}^{(n-1)}}$.

\beg{lem} \label{L2.1} Assume \textbf{\upshape(H1)}-\textbf{\upshape(H4)}. Then for
every $n\in\N$, the SDE $\eqref{EN}$ has a unique strong solution $X^{(n)}_{s,t}$ with
\beq\label{*2} \E\sup_{t\in [s,T]} |X^{(n)}_{s,t}|^\theta<\infty,\quad T>s, n\in\N.\end{equation} Moreover, for any $T>0$, there exists $t_0>0$ such that for all $s\in [0,T]$ and $X_{s,s}\in L^\theta(\OO\to\mathbb{R}^d;\scr F_s)$,
$$
\E \sup_{ t\in [s, s+t_0]} |X^{(n+1)}_{s,t}-X^{(n)}_{s,t}|^\theta\le 2^\theta \e^{-n}  \E\sup_{t\in [s,s+t_0]} |X^{(1)}_{s,t}|^\theta,\quad n\in\N.
$$
\end{lem}

\beg{proof} We only need to apply
Proposition \ref{LT} in the appendix and use the argument
in \cite[proof of Lemma 2.3]{W16} to
obtain the desired assertions. Here we omit
the details to save space.
\end{proof}

\begin{proof}[Proof of Theorem \ref{TEU}]

(1) First, we prove the existence of strong solution in $\scr P_\theta$.
For simplicity, we only consider  $s=0$ and denote
$X_{0,t}=X_t$, $t\ge 0$.

For $t>0$, let $\D_t$ be the family of all $\R^d$-valued c\`{a}dl\`{a}g functions on $[0,t]$ equipped
with the uniform norm. Since $\D_t$ is a Banach
space, so is $L^\theta(\Omega; \D_t)$.
Let $(X_t)_{t\in [0,t_0]}$ be the unique limit of $(X_t^{(n)})_{t\in [0,t_0]}$ in  Lemma \ref{L2.1}. Then  $(X_t)_{t\in [0,t_0]}$ is an adapted c\`{a}dl\`{a}g process  and  satisfies
\beq\label{A01}\lim_{n\to\infty} \sup_{t\in [0,t_0]} \W_\theta(\mu_t^{(n)},\L_{X_t})^\theta\le \lim_{n\to\infty} \E \sup_{t\in [0,t_0]}  |X^{(n)}_t- X_t|^\theta=0.
\end{equation}
Reformulate \eqref{EN} as
$$X^{(n)}_t=X_0+ \int_0^t  b(s,X^{(n)}_s,\mu_s^{(n-1)})\,\d s +\int_0^t\si(s)\,\d Z_s. $$
Now \eqref{A01}, \textbf{(H2)}, the local boundedness of $b$, and the dominated convergence theorem imply that $\P$-a.s.
$$X_t= X_0+\int_0^t b(s,X_s, \L_{X_s})\,\d s +\int_0^t\si(s)\,\d Z_s,\quad t\in [0,t_0].$$
Moreover, \eqref{*2} and \eqref{A01} lead to $ \E \sup_{s\in [0,t_0]} |X_s|^\theta<\infty$. Therefore, $(X_t)_{t\in [0,t_0]}$ solves \eqref{E1} up to time $t_0$.
The same assertion holds for
$(X_{s,t})_{t\in [s,(s+t_0)\land T]}$ and $s\in [0,T]$.
By solving the equation piecewise in time, and using the arbitrariness of $T>0$,
  we conclude that  \eqref{E1} has a unique strong solution
  $(X_t)_{t\ge 0}$ with
  $$ \E \sup_{s\in [0,t]}|X_s|^\theta<\infty,\ \ \ t\geq 0.$$

Next, we prove strong uniqueness in $\scr P_\theta$.
Let $X_t$ and $Y_t$ be two solutions to \eqref{E1} with $X_0=Y_0$ and $\E|X_t|^\theta+\E|Y_t|^\theta<\infty$, $t\geq0$.
It follows from \textbf{(H3)} that
$$
\d |X_t-Y_t|^2\leq\kappa_1(t)|X_t-Y_t|^2\,\d t+\kappa_2(t)\W_\theta(\L_{X_t},\L_{Y_t})|X_t-Y_t|\,\d t.
$$
For any $\varepsilon>0$, it is easy to see that
\begin{equation}\label{df334d}
\begin{aligned}
    \d(|X_t-Y_t|^2+\varepsilon)^{\frac{\theta}{2}}
    &=\frac{\theta}{2}\,
    (|X_t-Y_t|^2+\varepsilon)^{\frac{\theta-2}{2}}\,\d|X_t-Y_t|^2\\
    &\leq\frac{\theta}{2}\,\kappa_1(t)
(|X_t-Y_t|^2+\varepsilon)^\frac{\theta-2}{2}|X_t-Y_t|^2\,\d t\\
&\quad+\frac{\theta}{2}\,\kappa_2(t)\W_\theta(\L_{X_t},\L_{Y_t})
(|X_t-Y_t|^2+\varepsilon)^{\frac{\theta-1}{2}}\,\d t.
\end{aligned}
\end{equation}
Using the following inequality
\begin{equation}\label{elementary1}
yz^{\rho-1}\leq\frac{1}{\rho}\,y^{\rho}+\frac{\rho-1}{\rho}\,
z^\rho,\quad y,z\geq0,\,\rho\geq1
\end{equation}
with $\rho=\theta$, we get
\begin{align*}
(|X_t-Y_t|^2+\varepsilon)^{\frac{\theta}{2}}
    &\leq(|X_0-Y_0|^2+\varepsilon)^{\frac{\theta}{2}}+\int_0^t\frac{\theta}{2}\,\kappa_1(s)
(|X_s-Y_s|^2+\varepsilon)^\frac{\theta-2}{2}|X_s-Y_s|^2\,\d s
    \\&\quad+\int_0^t\frac{\theta-1}{2}\kappa_2(s)(|X_s-Y_s|^2+\varepsilon)^{\frac{\theta}{2}}\,\d s+\int_0^t\frac12\,\kappa_2(s)\W_\theta(\L_{X_s},\L_{Y_s})
    ^\theta\,\d s.
\end{align*}
Letting $\varepsilon\downarrow0$, using the fact that $\W_\theta(\L_{X_s},\L_{Y_s})^\theta\le
\E|X_s-Y_s|^\theta$, and
taking expectations on both sides, we obtain
\begin{align}\label{X-Y}
\E|X_t-Y_t|^\theta\leq\E|X_0-Y_0|^\theta+\frac{\theta}{2}\int_0^t(\kappa_1(s)+\kappa_2(s))\E|X_s-Y_s|^\theta
\,\d s,
\end{align}
which, together with Gronwall's inequality and $X_0=Y_0$, implies that
$$\E|X_t-Y_t|^\theta=0\quad \text{for all $t\geq0$}.$$
Thus, strong uniqueness in $\scr P_\theta$ for \eqref{E1} follows.

(2) Let   $(X_t)_{t\ge 0}$ solve \eqref{E1} with $\scr L_{X_0}=\mu_0$, and let   $(\tt X_t,\tt Z_t)_{t\ge 0}$ on
$(\tt\OO, \{\tt\F_t\}_{t\ge 0}, \tt\P)$ be a weak solution of \eqref{E1}  such that  $\L_{\tt X_0}|_{\tt\P}=\mu_0$, i.e.
$(\tt X_t)_{t\ge 0}$ solves
\begin{equation}\label{j4dd2s} \d \tt X_t = b(t,\tt X_t, \L_{\tt X_t}|_{\tt\P})\,\d t + \si(t)\,\d \tt Z_t,\quad \scr L_{\tt X_0}|_{\tt\P}=\mu_0.\end{equation}
Moreover, $\L_{X_t|\P}, \L_{\tilde{X}_t|\tilde{\P}}\in \scr P_{\theta}, t\geq 0$. We need to   prove
 $\L_{X_t}|_{\P}=\L_{\tt X_t}|_{\tt\P}$ for all $t\geq0$.
  Let $\mu_t= \L_{X_t}|_{\P}$ and
$$\bar{b}(t,x)= b(t,x, \mu_t),\quad t\geq0,x\in\R^d.$$
  By \textbf{\upshape(H1)}-\textbf{\upshape(H4)} and Proposition \ref{LT} below,
  the following SDE
\beq\label{E10} \d  \bar{X}_t = \bar{b}(t,\bar{X}_t)\,\d t
+ \si(t)\,\d \tt Z_t,\ \  \bar{X}_0= \tt X_0 \end{equation}
has a unique strong solution.
According to Yamada--Watanabe's theory for SDEs driven by
jump processes (cf. \cite[Theorem 1]{BLP15}),
it also satisfies weak uniqueness. Noting that
$$\d X_t= \bar{b}(t,X_t)\,\d t + \si(t)\,\d  Z_t,\quad \L_{X_0}|_{\P}= \L_{ \bar{X}_0}|_{\tt\P},$$ the weak uniqueness of \eqref{E10} implies
\beq\label{HW} \L_{\bar{X}_t}|_{\tt\P}= \L_{X_t}|_{\P}=\mu_t,\quad t\geq0.\end{equation}
So, \eqref{E10} can be rewritten as
$$ \d \bar{X}_t = b(t,\bar{X}_t,
\L_{\bar{X}_t}|_{\tt\P})\,\d t + \si(t)\,\d \tt Z_t,\quad \bar{X}_0=\tt X_0.$$
Since it follows from (1) that \eqref{j4dd2s}
has a strong well-posedness in $\scr P_{\theta}$,
we know that $\bar{X}=\tt X$. Therefore,  \eqref{HW}
implies $\L_{\tt X_t}|_{\tt \P} = \L_{X_t}|_{\P}$ for
all $t\geq0$, as required.
\end{proof}

For $\mu_0\in\scr P_\theta$, let $X_t(\mu_0)$ be the solution to \eqref{E1} with $\L_{X_0}=\mu_0$.
Let $P^\ast_t\mu_0$ be the distribution of  $X_t(\mu_0)$.

\begin{prp}\label{Wt}
Assume \textbf{\upshape(H1)}-\textbf{\upshape(H4)}.
For any $\mu_0,\nu_0\in\scr P_\theta$,
\begin{align}\label{wtp}
\mathbb{W}_\theta(P_t^\ast \mu_0,P_t^\ast \nu_0)\leq \exp\left[
\frac{1}{2}\int_0^t
   \left\{\kappa_1(s)+\kappa_2(s)\right\}
   \,\d s\right]
\mathbb{W}_\theta(\mu_0,\nu_0),\quad t\geq0.
\end{align}
\end{prp}

\begin{proof}
It follows from \eqref{X-Y} and Gronwall's inequality that
$$\E|X_t-Y_t|^\theta\leq \E|X_0-Y_0|^\theta
\exp\left[
\frac{\theta}{2}\int_0^t
   \left\{\kappa_1(s)+\kappa_2(s)\right\}
   \,\d s\right].$$
For any $\mu_0,\nu_0\in \scr P_\theta$, we can take $\F_0$-measurable random variables $X_0$ and $Y_0$ such that $\L_{X_0}=\mu_0$, $\L_{Y_0}=\nu_0$ and
$\W_\theta(\mu_0, \nu_0)^\theta=\E|X_0-Y_0|^\theta$.
Combining this with
$\W_\theta(P_t^*\mu_0, P_t^*\nu_0)^\theta\le
\E|X_t(\mu_0)-Y_t(\nu_0)|^\theta$, we obtain the desired
assertion.
\end{proof}

\section{Harnack inequalities}

In this section, we study the Harnack inequality
for \eqref{sbb}. In this case, the L\'{e}vy
noise $(Z_t)_{t\geq0}$ is
given by subordinate Brownian motion
$(W_{S_t})_{t\geq0}$, where $W=\{W_t\}_{t\geq 0}$
is a standard Brownian motion on $\R^d$,
and $S=\{S_t\}_{t\geq 0}$ is an independent
subordinator with Bernstein
function (Laplace exponent) $\phi$ given by \eqref{bern}.
Since the L\'{e}vy measure of $Z_t=W_{S_t}$ is
$$
    \nu_Z(\d x)=\int_{(0,\infty)}(2\pi s)^{-d/2}
    \e^{-|x|^2/(2s)}\,\nu_S(\d s)\,\d x,
$$
where $\nu_S$ is the L\'{e}vy measure of subordinator
$S$, it is not hard to verify that
\textbf{(H1)} is equivalent to
\beg{enumerate}
\item[\textbf{(H1$'$)}] $\int_{(1,\infty)}x^{\theta/2}\,\nu_S(\d x)<\infty$.
\end{enumerate}

\begin{rem}
    We list here some typical examples for Bernstein
    function $\phi$
    satisfying \textbf{\upshape(H1$'$)}.
    \begin{itemize}
        \item   \textup{(}Stable subordinators\textup{)} Let $\phi(r)=r^{\alpha}$ with drift $\varrho=0$
            and L\'{e}vy measure $\nu_S(\d x)=
            \frac{\alpha}{\Gamma(1-\alpha)}\,x^{-1-\alpha}
            \,\d x$, where $1/2<\alpha<1$. Then
            \textbf{\upshape(H1$'$)} holds
            if $1\leq\theta<2\alpha$;

        \item \textup{(}Relativistic stable subordinators\textup{)} Let $\phi(r)
    =(r+m^{1/\alpha})^{\alpha}-m$ with drift
    $\varrho=0$ and L\'{e}vy measure
    $\nu_S(\d x)=\frac{\alpha}{\Gamma(1-\alpha)}\,
    \e^{-m^{1/\alpha}x}x^{-1-\alpha}\,\d x$,
    where $0<\alpha<1$ and $m>0$. Then
            \textbf{\upshape(H1$'$)} holds for all  $\theta\geq1$;

        \item \textup{(}Gamma subordinators\textup{)} Let $\phi(r)=\log(1+r/a)$ with drift $\varrho=0$
            and L\'{e}vy measure
    $\nu_S(\d x)=x^{-1}\e^{-ax}\,\d x$, where $a>0$.
    Then \textbf{\upshape(H1$'$)} holds for all  $\theta\geq1$;

        \item   Let $\phi(r)=r\log(1+a/r)$ with
            drift $\varrho=0$ and L\'{e}vy measure
    $\nu_S(\d x)=x^{-2}(1-\e^{-ax}(1+ax))\,\d x$, where $a>0$.
    Then \textbf{\upshape(H1$'$)} holds if  $1\leq\theta<2$;

    \item   Let $\phi(r)=r\e^r\int_1^\infty\e^{-ry}y^{-n}\,
    \d y$ with drift $\varrho=0$ and L\'{e}vy
    measure $\nu_S(\d x)=n(1+x)^{-n-1}\,\d x$, where $n\in\N$.
    Then \textbf{\upshape(H1$'$)} holds if  $1\leq\theta<2n$.
    \end{itemize}
    We refer to \cite[Chapter 16]{SSV12} for an extensive
    list of such Bernstein functions.
\end{rem}

Moreover, we need the following assumption on $\sigma$:
\beg{enumerate}
\item[\textbf{(H5)}] For any $t\geq 0$, $\sigma(t)$ is invertible and there exists a non-decreasing function $\lambda:[0,\infty)\to[0,\infty)$ such that $$\|\sigma(t)^{-1}\|\leq \lambda(t),\quad t\geq0.$$
\end{enumerate}

For $t>0$, let
$$
K_1(t):=\exp\left[-\int_0^t\kappa_1(r)\,\d r
\right],
$$
and
$$
K(t,\theta):=\frac12\int_0^t\exp\left[
\frac{\theta}{2}\big\{\kappa_1(s)+\kappa_2(s)\big\}
-\frac12\int_0^s\kappa_1(r)\,\d r
\right]\kappa_2(s)\,\d s,
$$
where $\kappa_1$ and $\kappa_2$ are from \textbf{(H3)}.

Under \textbf{\upshape(H1$'$)} and \textbf{\upshape(H2)}-\textbf{\upshape(H5)}, it follows
from Theorem \ref{TEU} that for $\mu_0\in\scr P_\theta$,
equation \eqref{sbb} with $\L_{X_0}=\mu_0$ has a unique
solution $X_t(\mu_0)$. Define
$$P_t f(\mu_0):=\mathbb{E}f(X_t(\mu_0)),
\quad t\geq0,f\in\B_b(\mathbb{R}^d).$$
Note that, in general, $(P_t)_{t\geq0}$ is not a semigroup,
see \cite{W16}.

The main result in this section is the following theorem.

\begin{thm}\label{T3.2} Assume \textbf{\upshape(H1$'$)} and \textbf{\upshape(H2)}-\textbf{\upshape(H5)}.

  \smallskip\noindent\textup{(1)}
  For any    $ \mu_0,\nu_0\in \scr P_{\theta}$,
  $T>0$, and $f\in \B_b(\mathbb{R}^d)$ with $f\geq1$,
\begin{align*}
&P_T\log f(\nu_0)\\
&\leq\log P_Tf(\mu_0)
    +\lambda(T)^2\left\{\mathbb{W}_2(\mu_0,\nu_0)^2
    +K(T,\theta)^2\mathbb{W}_\theta(\mu_0,\nu_0)^2\right\}
    \E\left(
    \int_0^{T}K_1(s)
\,\d S_s
    \right)^{-1}.
\end{align*}

\smallskip\noindent\textup{(2)}
For any  $p>1$,  $ \mu_0,\nu_0\in \scr P_\theta $,    $\F_0$-measurable random variables
$X_0, Y_0$ with $\L_{X_0}=\mu_0, \L_{Y_0}=\nu_0$, $T>0$,
and non-negative
$f\in \B_b(\mathbb{R}^d)$,
\begin{align*}
    &\big(P_Tf(\nu_0)\big)^p
    \leq P_Tf^p(\mu_0)\\
   & \qquad\times\left(\E
    \exp\left[\frac{p\lambda(T)^2}{(p-1)^2}
    \left\{|X_0-Y_0|^2
+K(T,\theta)^2\mathbb{W}_\theta(\mu_0,\nu_0)^2\right\}
    \left(
    \int_0^{T}K_1(s)
\,\d S_s
    \right)^{-1}
    \right]
    \right)^{p-1}.
\end{align*}
 \end{thm}

For $\mu_0,\nu_0\in\scr P_\theta$ and $t>0$, let
$\mu_t:=P_t^\ast \mu_0$ and $ \nu_t:=P_t^\ast\nu_0$.
The following corollary is a direct consequence of
Theorem \ref{T3.2}, see \cite[Theorem 1.4.2]{Wbook}.

\beg{cor}\label{C1.2} Assume \textbf{\upshape(H1$'$)} and \textbf{\upshape(H2)}-\textbf{\upshape(H5)}. \
Let $\mu_0,\nu_0\in\scr P_{\theta\vee2}$ and $T>0$.
If $\E S_T^{-1}<\infty$, then $\mu_T$ and $\nu_T$
are equivalent. Furthermore, the following assertions hold.

\smallskip\noindent\textup{(1)}
It holds that
$$
\int_{\R^d}\log\left(\ff{\d\nu_T}{\d\mu_T}\right)\,
 \d\nu_T \le \left\{\mathbb{W}_2(\mu_0,\nu_0)^2
    +K(T,\theta)^2\mathbb{W}_\theta(\mu_0,\nu_0)^2\right\}
    \E\left(
    \int_0^{T}K_1(s)
\,\d S_s
    \right)^{-1}.
$$

\smallskip\noindent\textup{(2)}
For any  $p>1$ and   $\F_0$-measurable random variables
$X_0, Y_0$ with $\L_{X_0}=\mu_0, \L_{Y_0}=\nu_0$,
\begin{align*}
& \int_{\R^d}\left(\ff{\d\nu_T}{\d\mu_T}\right)^{1/(p-1)}
 \,\d\nu_T\\
 &\le\E
    \exp\left[\frac{p}{(p-1)^2}
    \left\{|X_0-Y_0|^2
+K(T,\theta)^2\mathbb{W}_\theta(\mu_0,\nu_0)^2\right\}
    \left(
    \int_0^{T}K_1(s)
\,\d S_s
    \right)^{-1}
    \right].
\end{align*}
\end{cor}

\subsection{Harnack inequalities under deterministic time-change}

Let $\ell:[0,\infty)\rightarrow[0,\infty)$ be
a sample path of
subordinator $S$, which is a non-decreasing and
c\`{a}dl\`{a}g function with $\ell(0)=0$.
For $\mu_0\in \scr P_\theta$, let $X_t(\mu_0)$
be the solution to \eqref{E1}
with $\L_{X_0}=\mu_0$.
By \textbf{(H2)} and \textbf{(H3)},
$b(t,\cdot,\L_{X_t(\mu_0)})$ is continuous and satisfies the one-sided
Lipschitz condition
$$
2\langle  b(t,x,\L_{X_t(\mu_0)})- b(t,y,\L_{X_t(\mu_0)}), x-y\rangle
 \le \kappa_1(t) |x-y|^2,\quad t\geq0,x,y\in\R^d.
$$
Thus, for any $\mu_0\in \scr P_\theta$,
the following SDE has a unique non-explosive
solution with $\L_{X_0^\ell}=\mu_0$:
\begin{equation}\label{jddc23}
    \d X_t^\ell=b(t,X_t^\ell,\L_{X_t(\mu_0)})\,\d t
    +\sigma(t)\,\d W_{\ell_t}.
\end{equation}
We denote the solution by $X_t^\ell(\mu_0)$.
The associated Markov operator is defined by
\begin{equation}\label{semig}
P_t^\ell f(\mu_0):=\E f\big(X_t^\ell(\mu_0)\big),
\quad t\geq0,f\in\scr B_b(\mathbb{R}^d),
\mu_0\in\scr P_\theta.
\end{equation}

\begin{prp}\label{timechange} Assume
\textbf{\upshape(H1$'$)} and \textbf{\upshape(H2)}-\textbf{\upshape(H5)}.

  \smallskip\noindent\textup{(1)}
  For any    $ \mu_0,\nu_0\in \scr P_\theta $,  $T>0$,
and $f\in \B_b(\mathbb{R}^d)$ with $f\geq1$, it holds
$$
P^{\ell}_T\log f(\nu_0)\leq\log P_T^{\ell} f(\mu_0)
    +\lambda(T)^2\left\{\mathbb{W}_2(\mu_0,\nu_0)^2
    +K(T,\theta)^2\mathbb{W}_\theta(\mu_0,\nu_0)^2\right\}
    \left(
    \int_0^{T}K_1(s)
\,\d \ell_s
    \right)^{-1}.
$$

\smallskip\noindent\textup{(2)}
For any    $p>1$, $ \mu_0,\nu_0\in \scr P_\theta $,    $\F_0$-measurable random variables
$X_0, Y_0$ with $\L_{X_0}=\mu_0, \L_{Y_0}=\nu_0$, $T>0$,
and non-negative
$f\in \B_b(\mathbb{R}^d)$, we have
\begin{align*}
    \big(P_T^{\ell}f(\nu_0)\big)^p
    &\leq P_T^{\ell}f^p(\mu_0)
    \cdot
    \left(\E
    \exp\left[\frac{p\lambda(T)^2}{(p-1)^2}
    \,|X_0-Y_0|^2
    \left(
    \int_0^{T}K_1(s)
\,\d \ell_s
    \right)^{-1}
    \right]
    \right)^{p-1}\\
    &\qquad\times
    \exp\left[\frac{p\lambda(T)^2}{p-1}
    \,K(T,\theta)^2\mathbb{W}_\theta(\mu_0,\nu_0)^2
    \left(
    \int_0^{T}K_1(s)
\,\d \ell_s
    \right)^{-1}
    \right].
\end{align*}
 \end{prp}

Following the line of \cite{WW,ZXC,DH},
for $\varepsilon\in(0,1)$, consider the
following regularization of $\ell$:
$$\ell^\varepsilon_t:=\frac{1}{\varepsilon}
\int_{t}^{t+\varepsilon}\ell_s\,\d s+\varepsilon t
=\int_0^1\ell_{\varepsilon s+t}\,\d s+\varepsilon t,
\quad t\geq0.$$
It is clear that, for each $\varepsilon\in(0,1)$,
the function $\ell^\varepsilon$ is absolutely
continuous, strictly increasing and satisfies
for any $t\geq0$
\begin{equation}\label{approximation}
    \ell^\varepsilon_t\downarrow\ell_t\quad
    \text{as $\varepsilon\downarrow0$}.
\end{equation}
For $\mu_0\in\scr P_\theta$,
let $X_t^{\ell^\varepsilon}(\mu_0)$ be the
solution to the following SDE with $\L_{X_0^{\ell^\varepsilon}}=\mu_0$:
$$
\d X^{\ell^\varepsilon}_t=b(t,X^{\ell^\varepsilon}_t,
\L_{X_t(\mu_0)})\,\d t
    +\sigma(t)\,\d
    W_{\ell^\varepsilon_t-\ell^\varepsilon_0}.
$$
Define the associated Markov operator
$P_t^{\ell^\varepsilon}$ by \eqref{semig}
with $\ell$ replaced by $\ell^\varepsilon$.

\begin{lem}\label{L3.3} Fix $\varepsilon\in(0,1)$ and
assume \textbf{\upshape(H1$'$)} and \textbf{\upshape(H2)}-\textbf{\upshape(H5)}. Then the assertions in Proposition in \ref{timechange} hold
with $\ell$ replaced by $\ell^\varepsilon$.
\end{lem}

\begin{proof} Fix $T>0$. Take $\F_0$-measurable random variables
$X_0, Y_0$ with $\L_{X_0}=\mu_0, \L_{Y_0}=\nu_0$.
Let $Y_t$ solve the SDE
\begin{equation}\begin{split}\label{EY}
\d Y_t&=b(t,Y_t,\L_{X_t(\nu_0)})\,\d t+\xi(t)\I_{[0,\tau)}(t) \frac{X^{\ell^\varepsilon}_t(\mu_0)-Y_t}
{|X^{\ell^\varepsilon}_t(\mu_0)-Y_t|}\,\d \ell^\varepsilon_t+\sigma(t)\,\d W_{\ell^\varepsilon_t-
\ell^\varepsilon_0}
\end{split}\end{equation}
with $\L_{Y_0}=\nu_0$, where
$$\tau:=T\wedge\inf\{t\geq 0;
X^{\ell^\varepsilon}_t(\mu_0)=Y_t\}$$
and
$$\xi(t):=
\big\{|X_0-Y_0|
+K(t,\theta)\mathbb{W}_\theta(\mu_0,\nu_0)\big\}
\frac{\sqrt{K_1(t)}}
{\int_0^{T}K_1(s)
\,\d \ell^\varepsilon_s}.
$$
It is clear that $(X^{\ell^\varepsilon}_t,Y_t)$ is well defined for $t<\tau$. By \textbf{(H3)}, it follows
that for $t<\tau$
$$
\d |X^{\ell^\varepsilon}_t(\mu_0)-Y_t|\leq \frac12\,\kappa_1(t)|X^{\ell^\varepsilon}_t(\mu_0)-Y_t|\,\d t+\frac12\,\kappa_2(t)\mathbb{W}_\theta(\mu_t,\nu_t)\,\d t-\xi(t)\,\d \ell^\varepsilon_t.
$$
Thus, by \eqref{wtp}, we obtain that
\begin{align*}
&\sqrt{K_1(t)}\,
|X^{\ell^\varepsilon}_t(\mu_0)-Y_t|\\
&\leq|X_0-Y_0|+\frac12\int_0^t
\sqrt{K_1(s)}\,
\kappa_2(s)\mathbb{W}_\theta(\mu_s,\nu_s)
\,\d s-\int_0^t
\sqrt{K_1(s)}\,
\xi(s)\,\d \ell^\varepsilon_s\\
&\leq|X_0-Y_0|+K(t,\theta)\mathbb{W}_\theta(\mu_0,\nu_0)
-\int_0^t
\sqrt{K_1(s)}\,\xi(s)\,\d \ell^\varepsilon_s\\
&=\big\{|X_0-Y_0|
+K(t,\theta)\mathbb{W}_\theta(\mu_0,\nu_0)\big\}
\left\{1-
\frac{\int_0^t
K_1(s)
\,\d \ell^\varepsilon_s
}{\int_0^{T}K_1(s)
\,\d \ell^\varepsilon_s}
\right\}
\end{align*}
for all $t<\tau$. If $\tau(\omega)>T$ for some $\omega\in\Omega$,
we can take $t=T$ in the above inequality
to get
$$
0<\sqrt{K_1(T)}\,|X^{\ell^\varepsilon}_T
(\mu_0,\omega)-Y_T(\omega)|\leq0,
$$
which is absurd. Therefore, $\tau\leq T$. Letting $Y_t:=X^{\ell^\varepsilon}_t(\mu_0)$ for $t\in[\tau,T]$, then $Y_t$ solves \eqref{EY} for $t\in[\tau,T]$. In particular, $X^{\ell^\varepsilon}_T(\mu_0)=Y_T$.

Denote by $\gamma^\varepsilon:[\ell^\varepsilon_0,\infty)
\rightarrow[0,\infty)$ the inverse function of $\ell^\varepsilon$. Then $\ell^\varepsilon_
{\gamma^\varepsilon_t}=t$ for $t\geq\ell^\varepsilon_0$,
$\gamma^\varepsilon_
{\ell^\varepsilon_t}=t$ for $t\geq0$, and $t\mapsto\gamma^\varepsilon_t$
is absolutely continuous and
strictly increasing. Let
$$
\widetilde{W}_t:=\int_{0}^{t}\Psi(r)\,\d r+
W_t\quad\text{and}\quad M_t:=-\int_0^t
\< \Psi(r), \d
W_r\>,\quad t\geq0,
$$
where $$\Psi(r):=\sigma^{-1}
_{\gamma^\varepsilon_{r+\ell^\varepsilon_0}}
\Phi\big(\gamma^\varepsilon_{r+\ell^\varepsilon_0}\big)
\quad\text{and}\quad
\Phi(r):=\xi(r)\I_{[0,\tau)}(r)
\frac{X^{\ell^\varepsilon}_r(\mu_0)
-Y_r}{|X^{\ell^\varepsilon}_r(\mu_0)-Y_r|}.
$$

By \textbf{(H5)} and the elementary inequality
that $(a+b)^2\leq2a^2+2b^2$ for $a,b\geq0$,
the compensator of the martingale $M_t$
satisfies, for $t\geq0$,
\begin{equation}\begin{split}\label{ddfaw}
    \<M\>_t&=\int_0^t|\Psi(r)|^2\,\d r
    \leq\int_0^T|\sigma_s^{-1}\Phi(s)|^2\,\d\ell^\varepsilon_s\\
    &\leq\int_0^T\lambda(s)^2\Phi(s)^2\,\d\ell^\varepsilon_s
    \leq\lambda(T)^2\int_0^T\xi(s)^2\,\d\ell^\varepsilon_s\\
    &\leq\lambda(T)^2
    \big\{|X_0-Y_0|
+K(T,\theta)\mathbb{W}_\theta(\mu_0,\nu_0)\big\}^2
\left(
    \int_0^{T}K_1(s)
\,\d \ell^\varepsilon_s
    \right)^{-1}\\
&\leq2\lambda(T)^2
    \left\{|X_0-Y_0|^2
    +K(T,\theta)^2\mathbb{W}_\theta(\mu_0,\nu_0)^2\right\}
    \left(
    \int_0^{T}K_1(s)
\,\d \ell^\varepsilon_s
    \right)^{-1}.
\end{split}\end{equation}

By Novikov's criterion, we have $\E[R|\F_0]=1$, where
$$
R:=\exp\left[M_{\ell^\varepsilon_T-\ell^\varepsilon_0}
-\frac12\<M\>_{\ell^\varepsilon_T-\ell^\varepsilon_0}
\right].
$$
According to Girsanov's theorem,
$(\widetilde{W}_t)_{0\leq t\leq\ell^\varepsilon(T)-\ell^\varepsilon(0)}$ is a
$d$-dimensional Brownian motion under the new probability
measure $R\P(\cdot|\F_0)$. Rewrite \eqref{EY} as
$$
    \d Y_t=b(t,Y_t,\L_{X_t(\nu_0)})\,\d t+\sigma(t)\,\d \widetilde{W}_{\ell^\varepsilon_t-
\ell^\varepsilon_0}.
$$
Thus, the distribution of $(Y_t)_{0\leq t\leq T}$
under $R\P(\cdot|\F_0)$
coincides with that of
$(X^{\ell^\varepsilon}_t(\nu_0))_{0\leq t\leq T}$
under $\P(\cdot|\F_0)$; in particular, it holds that for
any $f\in \B_b(\R^d)$,
\begin{equation}\label{jh43cdf}
    \E_{\P(\cdot|\F_0)}f(X^{\ell^\varepsilon}_T(\nu_0))=
    \E_{R\P(\cdot|\F_0)}f(Y_T)=
    \E_{\P(\cdot|\F_0)}\left[Rf(Y_T)\right]
    =\E_{\P(\cdot|\F_0)}\big[Rf(X^{\ell^\varepsilon}_T
    (\mu_0))\big].
\end{equation}
By \eqref{jh43cdf}, the Young
inequality (cf. \cite[p.\ 24]{Wbook}), and the observation that
\begin{align*}
    \log R&=-\int_0^{\ell^\varepsilon_T-
    \ell^\varepsilon_0}\<\Psi(r),\d W_r\>
    -\frac12\int_0^{\ell^\varepsilon_T-
    \ell^\varepsilon_0}|\Psi(r)|^2\,\d r\\
    &=-\int_0^{\ell^\varepsilon_T-
    \ell^\varepsilon_0}\<\Psi(r),\d \widetilde{W}_r\>
    +\frac12\,\<M\>_{\ell^\varepsilon_T-
    \ell^\varepsilon_0},
\end{align*}
we get that,
for any positive $f\in \B_b(\R^d)$,
\begin{align*}
\E_{\P(\cdot|\F_0)}\log f(X^{\ell^\varepsilon}_T(\nu_0))
&=\E_{\P(\cdot|\F_0)}\big[R\log f(X^{\ell^\varepsilon}_T(\mu_0))\big]\\
&\leq\log\E_{\P(\cdot|\F_0)} f(X^{\ell^\varepsilon}_T(\mu_0))
+\E_{\P(\cdot|\F_0)}[R\log R]\\
&=\log \E_{\P(\cdot|\F_0)} f(X^{\ell^\varepsilon}_T(\mu_0))+
\E_{R\P(\cdot|\F_0)}\log R\\
&=\log \E_{\P(\cdot|\F_0)} f(X^{\ell^\varepsilon}_T(\mu_0))+
\frac12\,\E_{R\P(\cdot|\F_0)}\<M\>_{\ell^\varepsilon_T
-\ell^\varepsilon_0}.
\end{align*}
Combining this with the Jensen inequality
and \eqref{ddfaw}, we obtain
\begin{align*}
    &P_T^{\ell^\varepsilon}\log f(\nu_0)\\
    &=\E\big\{
    \E_{\P(\cdot|\F_0)}\log f(X^{\ell^\varepsilon}_T(\nu_0))
    \big\}\\
    &\leq\log\E\big\{
    \E_{\P(\cdot|\F_0)} f(X^{\ell^\varepsilon}_T(\mu_0))
    \big\}+\frac12\,
    \E\big\{
    \E_{\P(\cdot|\F_0)}[R\<M\>_{\ell^\varepsilon_T-
    \ell^\varepsilon_0}]
    \big\}\\
    &\leq\log P_T^{\ell^\varepsilon} f(\mu_0)+\lambda(T)^2\left\{\E|X_0-Y_0|^2
    +K(T,\theta)^2\mathbb{W}_\theta(\mu_0,\nu_0)^2\right\}
    \left(
    \int_0^{T}K_1(s)
\,\d \ell^\varepsilon_s
    \right)^{-1}.
\end{align*}
Taking infimum over $\L_{X_0}=\mu_0, \L_{Y_0}=\nu_0$, we derive the log-Harnack inequality.

Next, we prove the power-Harnack inequality.
For any $p>1$ and positive $f\in \B_b(\R^d)$, we find
with \eqref{jh43cdf} and the H\"{o}lder inequality that
\begin{equation}\begin{split}\label{fds54fff}
    \E_{\P(\cdot|\F_0)}
    f(X_T^{\ell^\varepsilon}(\nu_0))
    &=\E_{\P(\cdot|\F_0)}\big[R
    f(X_T^{\ell^\varepsilon})(\mu_0)\big]\\
    &\leq \big(\E_{\P(\cdot|\F_0)}f^p
    (X_T^{\ell^\varepsilon})(\mu_0)\big)^{1/p}
    \cdot\big(
    \E_{\P(\cdot|\F_0)}\big[R^{p/(p-1)}\big]
    \big)^{(p-1)/p}.
\end{split}\end{equation}
Since by \eqref{ddfaw}
\begin{align*}
    R^{p/(p-1)}&=\exp\left[
    \frac{p}{p-1}\,M_{\ell^\varepsilon_T-\ell^\varepsilon_0}
    -\frac{p}{2(p-1)}\,
    \<M\>_{\ell^\varepsilon_T-\ell^\varepsilon_0}
    \right]\\
    &=\exp\left[
    \frac{p}{2(p-1)^2}\,\<M\>_{\ell^\varepsilon_T
    -\ell^\varepsilon_0}
    \right]\times
    \exp\left[
    \frac{p}{p-1}\,M_{\ell^\varepsilon_T-\ell^\varepsilon_0}
    -\frac{p^2}{2(p-1)^2}\,
    \<M\>_{\ell^\varepsilon_T-\ell^\varepsilon_0}
    \right]\\
    &\leq\exp\left[\frac{p\lambda(T)^2}{(p-1)^2}
    \left\{|X_0-Y_0|^2
    +K(T,\theta)^2\mathbb{W}_\theta(\mu_0,\nu_0)^2\right\}
    \left(
    \int_0^{T}K_1(s)
\,\d \ell^\varepsilon_s
    \right)^{-1}
    \right]\\
    &\quad\times\exp\left[
    \frac{p}{p-1}\,M_{\ell^\varepsilon_T-\ell^\varepsilon_0}
    -\frac{p^2}{2(p-1)^2}\,
    \<M\>_{\ell^\varepsilon_T-\ell^\varepsilon_0}
    \right],
\end{align*}
we know that
$$
    \E_{\P(\cdot|\F_0)}\left[R^{p/(p-1)}\right]
    \leq\exp\left[\frac{p\lambda(T)^2}{(p-1)^2}
    \left\{|X_0-Y_0|^2
    +K(T,\theta)^2\mathbb{W}_\theta(\mu_0,\nu_0)^2\right\}
    \left(
    \int_0^{T}K_1(s)
\,\d \ell^\varepsilon_s
    \right)^{-1}
    \right].
$$
Inserting this estimate into \eqref{fds54fff}, we
obtain
\begin{align*}
    P_T^{\ell^\varepsilon}f(\nu_0)
    &=\E\big\{\E_{\P(\cdot|\F_0)}
    f(X_T^{\ell^\varepsilon}(\nu_0))\big\}\\
    &\leq\E\left\{
    \big(\E_{\P(\cdot|\F_0)}f^p
    (X_T^{\ell^\varepsilon})(\mu_0)\big)^{1/p}
    \exp\left[\frac{\lambda(T)^2}{p-1}\,
    |X_0-Y_0|^2
    \left(
    \int_0^{T}K_1(s)
\,\d \ell^\varepsilon_s
    \right)^{-1}
    \right]\right\}\\
    &\quad\times\exp\left[\frac{\lambda(T)^2}{p-1}\,
    K(T,\theta)^2\mathbb{W}_\theta(\mu_0,\nu_0)^2
    \left(
    \int_0^{T}K_1(s)
\,\d \ell^\varepsilon_s
    \right)^{-1}
    \right].
\end{align*}
It remains to use the H\"{o}lder inequality
to get the desired power-Harnack inequality.
\end{proof}

The following two assumptions will be used:
\begin{enumerate}
        \item[\textbf{(A1)}]
            $\sigma$ is piecewise constant, i.e.\ there
            exists a sequence $\{t_n\}_{n\geq0}$
            with $t_0=0$ and $t_n\uparrow\infty$
            such that
            $$
                \sigma(t)=\sum_{n=1}^\infty
                \I_{[t_{n-1},t_n)}(t)
                \sigma(t_{n-1});
            $$

        \item[\textbf{(A2)}]
            For every $t>0$,
            there exists $C_t>0$ depending
            only on $t$ such that
            $$
                |b(s,x,\mu)-b(s,y,\mu)|
                \leq C_t|x-y|,
                \quad 0\leq s\leq t,\,
                x,y\in\R^d,\,\mu\in\scr P_\theta.
            $$
    \end{enumerate}

\begin{lem}\label{appro11}
    Assume \textbf{\upshape(H1$'$)} and \textbf{\upshape(H2)}-\textbf{\upshape(H5)}. If \textbf{\upshape(A1)} and \textbf{\upshape(A2)} hold,
    then for all $t\geq0$
    and $ \mu_0\in \scr P_\theta $,
    $$
        \lim_{\varepsilon\downarrow0}
        X_t^{\ell^\varepsilon}(\mu_0)=X_t^\ell(\mu_0)
        \quad \P\text{-a.s.}
    $$
\end{lem}

\begin{proof}
    It is not hard to obtain from \textbf{\upshape(A1)} and \textbf{\upshape(A2)} that, for all $t\geq0$ and $\varepsilon\in(0,1)$,
    $$
        |X_t^{\ell^\varepsilon}(\mu_0)-X_t^\ell(\mu_0)|
        \leq C_t\int_0^t|X_s^{\ell^\varepsilon}(\mu_0)
        -X_s^\ell(\mu_0)|
        \,\d s+g(\varepsilon,t),
    $$
    where
    $$
        g(\varepsilon,t):=
        \sup_{s\in[0,t]}\|\sigma(s)\|\cdot
        \left(
        \big|W_{\ell^{\varepsilon}_t-
        \ell^{\varepsilon}_0}
        -W_{\ell_t}\big|
        +2\sum_{n:\,t_n<t}
        \big|W_{\ell^{\varepsilon}_{t_n}-
        \ell^{\varepsilon}_0}
        -W_{\ell_{t_n}}\big|
        \right).
    $$
    Using Gronwall's inequality, we get
    $$
        |X_t^{\ell^\varepsilon}(\mu_0)-X_t^\ell(\mu_0)|
        \leq g(\varepsilon,t)
        +C_t\int_0^t g(\varepsilon,s)
        \e^{C_t(t-s)}\,\d s.
    $$
    Due to \eqref{approximation}, for all $s\geq0$, $\lim_{\varepsilon\downarrow 0}g(\varepsilon,s)=0$ a.s. It remains to use the dominated convergence theorem to finish the proof.
\end{proof}

\begin{proof}[Proof of Proposition \ref{timechange}]
    Fix $T>0$. By a standard
    approximation argument, we may and do
    assume that $f\in C_b(\R^d)$.

\emph{Step 1:}
    Assume \textbf{(A1)} and \textbf{(A2)}. Since
    $\ell_t$ is of bounded variation,
    it is not hard to verify from \eqref{approximation} that
    $$
        \lim_{\varepsilon\downarrow 0}
        \int_0^{T}K_1(s)\,\d \ell^\varepsilon_s
        =\int_0^{T}K_1(s)\,\d \ell_s.
    $$
    Letting
    $\varepsilon\downarrow0$ in Lemma \ref{L3.3},
    and using Lemma \ref{appro11},
    we get the desired inequalities.

\emph{Step 2:}
    Assume \textbf{(A2)}. Clearly, we can pick a sequence
    of $\R^d\otimes\R^d$-valued
    functions $\{\sigma_n\,:\,n\in\N\}$ on $[0,\infty)$
    such that each $\sigma_n$ is piecewise constant,
    $\|(\sigma_n(t))^{-1}\|\leq\lambda(t)$ for all $n\in\N$ and
    $t\in[0,T]$, and $\sigma_n\rightarrow\sigma$ in
    $L^2([0,T];\,\d\ell)$ as $n\rightarrow\infty$.
    Let $X_t^{\ell,n}$ solve \eqref{jddc23}
    with $\sigma$ replaced by $\sigma_n$ and
    $X_0^{\ell,n}=X_0^\ell$,
    and denote by
    $P_t^{\ell,n}$ the associated Markov operator.
    By Step 1, the statement of Proposition \ref{timechange}
    holds with $P_{T}^{\ell}$
    replaced by $P_{T}^{\ell,n}$.
    It suffice to prove that
    \begin{equation}\label{jh4dgc}
        \lim_{n\rightarrow\infty}P_T^{\ell,n}
        f=P_T^{\ell}f,\quad f\in C_b(\R^d).
    \end{equation}
    It follows
    from \textbf{(A2)} that
    $$
        |X_t^{\ell,n}-X_t^\ell|
        \leq C_t\int_0^t|X_s^{\ell,n}-X_s^\ell|
        \,\d s+\left|
        \int_0^t\left\{
        \sigma_n(s)-\sigma(s)
        \right\}\,\d W_{\ell_s}
        \right|.
    $$
    Noting that
    $$
        \lim_{n\rightarrow\infty}\E
        \left|
        \int_0^t\left\{
        \sigma_n(s)-\sigma(s)
        \right\}\,\d W_{\ell_s}
        \right|^2
        =\lim_{n\rightarrow\infty}
        \int_0^t
        \left\|
        \sigma_n(s)-\sigma(s)
        \right\|_{HS}^2\,\d\ell_s=0,
    $$
    we have (up to a subsequence) a.s. $\lim_{n\rightarrow\infty}\int_0^t\left\{
        \sigma_n(s)-\sigma(s)
        \right\}\,\d W_{\ell_s}=0$. Then
        as in the proof of Lemma \ref{appro11}, we
        find that for all $t\geq0$,
    $$
        \lim_{n\rightarrow\infty}
        X_t^{\ell,n}=X_t^\ell
        \quad \text{a.s.},
    $$
    which implies \eqref{jh4dgc}.

\emph{Step 3:}
    For the general case, we shall make use of
    the approximation argument
    in \cite[part (c) of proof of Theorem 2.1]{WW}
    (see also \cite{DH}). Let
    $$
        \tilde{b}(t,x,\mu_t):=b(t,x,\mu_t)
        -\frac12\,\kappa_1(t)x,\quad t\geq0,\,x\in\R^{d}.
    $$
    By {\bf(H3)}, it is easy to see
    that the mapping
    $\operatorname{id}-\varepsilon\tilde{b}(t,\cdot,\mu_t)
    :\R^d\rightarrow\R^d$ is injective
    for any $\varepsilon>0$ and $t\geq0$.
    For $\varepsilon>0$ and $t>0$, let
    $$
        \tilde{b}^{(\varepsilon)}(t,x,\mu_t)
        :=\frac1\varepsilon\left[
        \left(\operatorname{id}
        -\varepsilon\tilde{b}(t,\cdot,\mu_t)\right)^{-1}(x)-x
        \right],\quad x\in\R^d.
    $$
    Then for any $\varepsilon>0$ and $t>0$, $\tilde{b}^{(\varepsilon)}(t,\cdot,\mu_t)$
    is dissipative and satisfies \textbf{(A2)} with
    $b$ replaced by $\tilde{b}^{(\varepsilon)}$,
    $|\tilde{b}^{(\varepsilon)}(t,\cdot,\mu_t)|\leq
    |\tilde{b}(t,\cdot,\mu_t)|$ and
    $\lim_{\varepsilon\downarrow0}
    \tilde{b}^{(\varepsilon)}(t,\cdot,\mu_t)=
    \tilde{b}(t,\cdot,\mu_t)$. Let
    $b^{(\varepsilon)}(t,x,\mu_t):=
    \tilde{b}^{(\varepsilon)}(t,x,\mu_t)
    +\frac12\,\kappa_1(t)x$. Then $b^{(\varepsilon)}(t,\cdot,\mu_t)$ also satisfies
    \textbf{(A2)} with
    $b$ replaced by $b^{(\varepsilon)}$ and
    \begin{equation}\label{jg4dv7}
        2\<b^{(\varepsilon)}(t,x,\mu_t)-
        b^{(\varepsilon)}(t,y,\mu_t),
        x-y\>\leq \kappa_1(t)|x-y|^2.
    \end{equation}
    Let $X_t^{\ell,(\varepsilon)}(\mu_0)$
    solve the SDE \eqref{jddc23} with
    $b$ replaced by $b^{(\varepsilon)}$ and $X_0^{\ell,(\varepsilon)}=X_0^\ell$. Denote by
    $P_t^{\ell,(\varepsilon)}$ the associated Markov
    operator. Due to the
    second part of the proof, Proposition \ref{timechange}
    holds with $P_T^\ell$ replaced
    by $P_T^{\ell,(\varepsilon)}$. Then we only need
    to show that
    \begin{equation}\label{k54fc23f}
        \lim_{\varepsilon\downarrow0}P_T^{\ell,(\varepsilon)}
        f=P_T^{\ell}f,\quad
        f\in C_b(\R^d).
    \end{equation}
    To this end, we obtain from \eqref{jg4dv7}
    and \eqref{elementary1} with $\rho=2$ that
    \begin{align*}
        \d|X_t^{\ell,(\varepsilon)}
        -X_t^{\ell}|^2
        &=2\<X_t^{\ell,(\varepsilon)}
        -X_t^{\ell},
        b^{(\varepsilon)}(t,
        X_t^{\ell,(\varepsilon)},\mu_t)
        -b^{(\varepsilon)}
        (t,X_t^{\ell},\mu_t)\>\,\d t\\
        &\quad+2\<X_t^{\ell,(\varepsilon)}
        -X_t^{\ell},
        b^{(\varepsilon)}(t,X_t^{\ell},\mu_t)
        -b(t,X_t^{\ell},\mu_t)\>\,\d t\\
        &\leq\kappa_1(t)
        |X_t^{\ell,(\varepsilon)}
        -X_t^{\ell}|^2\,\d t
        +2|X_t^{\ell,(\varepsilon)}
        -X_t^{\ell}|\cdot
        \big|
        b^{(\varepsilon)}(t,X_t^{\ell},\mu_t)
        -b(t,X_t^{\ell},\mu_t)
        \big|\,\d t\\
        &\leq\big(\kappa_1(t)+1\big)^+
        |X_t^{\ell,(\varepsilon)}
        -X_t^{\ell}|^2\,\d t
        +\big|
        b^{(\varepsilon)}(t,X_t^{\ell},\mu_t)
        -b(t,X_t^{\ell},\mu_t)
        \big|^2\,\d t.
    \end{align*}
    This yields that
    $$
        |X_t^{\ell,(\varepsilon)}-X_t^{\ell}|^2
        \leq\int_0^t
        \big(\kappa_1(s)+1\big)^+
        |X_s^{\ell,(\varepsilon)}
        -X_s^{\ell}|^2\,\d s
        +\int_0^t
        \big|
        b^{(\varepsilon)}(s,X_s^{\ell},\mu_s)
        -b(s,X_s^{\ell},\mu_s)
        \big|^2\,\d s.
    $$
    Combining this with Gronwall's inequality,
    we obtain
    $$
        |X_t^{\ell,(\varepsilon)}
        -X_t^{\ell}|^2
        \leq\exp\left[
        \int_0^t
        \big(\kappa_1(s)+1\big)^+\,\d s
        \right]
        \cdot\int_0^t
        \big|
        b^{(\varepsilon)}(s,X_s^{\ell},\mu_s)
        -b(s,X_s^{\ell},\mu_s)
        \big|^2\,\d s.
    $$
    By {\bf (H2)} and {\bf(H4)}, letting $\varepsilon\downarrow0$ and using the
    dominated convergence theorem,
    we get $\lim_{\varepsilon\downarrow0}
    X_t^{\ell,(\varepsilon)}=X_t^{\ell}$
    for all $t\geq0$. In particular,
    \eqref{k54fc23f} holds. The proof
    is now finished.
\end{proof}

\subsection{Proof of Theorem \ref{T3.2}}

\begin{proof}[Proof of Theorem \ref{T3.2}]
   Since the processes $W$ and $S$ are independent,
    it holds that
   $$
        P_{T}f(\cdot)=\E\left[P_{T}^{\ell}f(\cdot)
        \left|_{\ell=S}\right.
        \right],\quad T>0,\,f\in\B_b(\R^d).
    $$
  Combining the estimates in Proposition \ref{timechange}
  with the Jensen inequality and the
  H\"{o}lder inequality, we obtain the
  desired Harnack type inequalities.
\end{proof}

\section{Appendix}\label{app}

The following result should be known, but we could not
find a reference and so we include a proof for the sake
of completeness.

\begin{prp}\label{LT} Assume that
$b:[0,\infty)\times\mathbb{R}^d\rightarrow
\mathbb{R}^d$ is measurable and continuous in the
space variable $x\in\mathbb{R}^d$ and
$\sigma:[0,\infty)\rightarrow
\mathbb{R}^d\otimes\R^d$ is measurable and
locally bounded. Let $(Z_t)_{t\geq0}$ be a
L\'{e}vy process with L\'{e}vy measure $\nu_Z$
satisfying $\int_{|x|\geq1}|x|^{\theta}\nu_Z(\d x)<\infty$ for some $\theta\geq1$. If there exists a locally bounded function $\kappa:[0,\infty)\to\mathbb{R}$
such that
\begin{align}\label{tb} 2\<b(t,x)-b(t,y)
,x-y\>\leq \kappa(t)|x-y|^2,\quad
x,y\in\mathbb{R}^d, t\geq0,
\end{align}
and $b(t,0)$ is locally bounded in the
time variable $t\geq0$, then the SDE
$$\d X_t=b(t,X_t)\,\d t+\sigma(t)\,\d Z_t
$$
starting from $\F_0$-measurable initial value $X_0$ with $\L_{X_0}\in\scr P_{\theta}$ has a unique strong solution satisfying
$$\E\sup_{s\in[0,t]}|X_t|^{\theta}<\infty
\quad\text{for all $t>0$.}
$$
\end{prp}

\begin{proof}
Under our assumptions, it is well known that
the SDE has a unique (strong) solution. It remains
to prove that the moments are finite.
Denote by $(l,Q,\nu_Z)$ the
L\'{e}vy triplet of $(Z_t)_{t\geq0}$.
By the L\'{e}vy-It\^{o}
decomposition (see e.g. \cite[Theorem 2.4.16]{App09}),
$$Z_t=lt+\sqrt{Q}W_t+\int_0^t\int_{|x|\geq1} x\,N(\d s,\d x)+\int_0^t\int_{|x|< 1} x\,\tilde{N}(\d s,\d x),$$
where $W=(W_t)_{t\geq0}$ is a $d$-dimensional (standard)
Brownian motion, $N$ is a Poisson random measure with intensity $\nu_Z(\d x)\d s$ and independent of  $W$, and $\tilde{N}$ is the associated
compensated Poisson random measure.
By It\^{o}'s formula (cf. \cite[Theorem 4.4.7]{App09}),
\begin{align*}
    \d |X_t|^2
    &=2\<X_{t-},b(t,X_{t-})+\sigma(t)l\>\,\d t
    +2\<X_{t-},\sigma(t)\sqrt{Q}\,\d W_t\>
    +\|\sigma(t)\sqrt{Q}\|_{HS}^2\,\d t\\
    &\quad+\int_{|x|\geq1}\left(|X_{t-}+\sigma(t)x|^2
    -|X_{t-}|^2\right)\,N(\d t, \d x)\\
    &\quad+\int_{|x|< 1}\left(|X_{t-}+\sigma(t)x|^2-|X_{t-}|^2\right)\,\tilde{N}(\d t, \d x)\\
    &\quad+\int_{|x|< 1}\left(|X_{t-}+\sigma(t)x|^2-|X_{t-}|^2
    -2\<X_{t-},\sigma(t)x\>\right)\,\nu_Z(\d x)\d t.
\end{align*}
Set $p:=\theta/2\geq1/2$. Applying It\^{o}'s formula again, we obtain
\begin{align*}
    &\d (1+|X_t|^2)^p
    =2p(p-1)(1+|X_{t-}|^2)^{p-2}
    |\big(\sigma(t)\sqrt{Q}\big)^*X_{t-}|^2
    \,\d t\\
    &\,\,\,+p(1+|X_{t-}|^2)^{p-1}\left(2\<X_{t-},
    b(t,X_{t-})+\sigma(t)l\>+\|\sigma(t)\sqrt{Q}\|_{HS}^2
    +\int_{|x|<1}|\sigma(t)x|^2\,\nu_Z(\d x)\right)\,\d t\\
    &\,\,\,+
    2p(1+|X_{t-}|^2)^{p-1}\<X_{t-}\,,\sigma(t)\sqrt{Q}\,
    \d W_t\>\\
    &\,\,\,+\int_{|x|\geq 1}
    J_1(x,t,p)\,N(\d t, \d x)
    +\int_{|x|< 1}J_1(x,t,p)\,\tilde{N}(\d t, \d x)
    +\int_{|x|< 1}J_2(x,t,p)\,\nu_Z(\d x)\d t,
\end{align*}
where
$$
   J_1(x,t,p):=(1+|X_{t-}+\sigma(t)x|^2)^p-(1+|X_{t-}|^2)^p,
$$
and
$$
    J_2(x,t,p):=
    (1+|X_{t-}+\sigma(t)x|^2)^p-(1+|X_{t-}|^2)^p
    -p(1+|X_{t-}|^2)^{p-1}(|X_{t-}+\sigma(t)x|^2
    -|X_{t-}|^2).
$$
Since $\sigma(t)$ and $b(t,0)$ are locally bounded
in $t\geq0$, it follows
from \eqref{tb} that we may find out a nondecreasing function
$H_1:[0,\infty)\rightarrow(0,\infty)$ such that
\begin{align*}
    \max\bigg\{&(1+|X_{t-}|^2)^{-1}|
    \big(\sigma(t)\sqrt{Q}\big)^*X_{t-}|^2,\\
    &2\<X_t,
    b(t,X_{t-})+\sigma(t)l\>+\|\sigma(t)\sqrt{Q}\|_{HS}^2
    +\int_{|x|<1}|\sigma(t)x|^2\,\nu_Z(\d x)
    \bigg\}
    \leq H_1(t)(1+|X_{t-}|^2).
\end{align*}
Then, we get
\begin{align*}
    (1+|X_s|^2)^p&\leq(1+|X_0|^2)^p+p(2p-1)H_1(s)\int_0^s
    (1+|X_{r-}|^2)^p\,\d r\\
    &\quad+2p\int_0^s
    (1+|X_{r-}|^2)^{p-1}\<X_{r-}\,,\sigma(r)\sqrt{Q}\,
    \d W_r\>
    +\int_0^s\int_{|x|\geq 1}
    J_1(x,r,p)\,N(\d r, \d x)\\
    &\quad+\int_0^s\int_{|x|< 1}J_1(x,r,p)\,
    \tilde{N}(\d r, \d x)
    +\int_0^s\int_{|x|< 1}J_2(x,r,p)\,\nu_Z(\d x)\d r\\
    &=:(1+|X_0|^2)^p+\sum_{i=1}^5I_i(s,p).
\end{align*}
Let $\tau_n:=\inf\{t\geq0: |X_t|\geq n\}$ for
$n\in\N$. Then we have
\begin{equation}\label{esti}
    \E\sup_{s\in[0,t\wedge\tau_n)}
    (1+|X_s|^2)^p\leq\E(1+|X_0|^2)^p
    +\sum_{i=1}^5\E\sup_{s\in[0,t\wedge\tau_n)}
    |I_i(s,p)|.
\end{equation}
We shall estimate these terms separately.
First,
\begin{align*}
    \E\sup_{s\in[0,t\wedge\tau_n)}I_1(s,p)
    &\leq p(2p-1)H_1(t)\E\int_0^{t\wedge\tau_n}
    (1+|X_{r-}|^2)^p\,\d r\\
    &\leq p(2p-1)H_1(t)\int_0^t
    \E\sup_{s\in[0,r\wedge\tau_n)}
    (1+|X_{s}|^2)^p\,\d r.
\end{align*}
By the Burkholder-Davis-Gundy inequality, there exist
a constant $c_1>0$ and a nondecreasing function
$H_2:[0,\infty)\rightarrow(0,\infty)$ such that
\begin{align*}
    \E\sup_{s\in[0,t\wedge\tau_n)}|I_2(s,p)|
    &\leq2c_1p\,\E\left(
    \int_0^{t\wedge\tau_n}(1+|X_{r-}|^2)^{2p-2}
    |\big(\sigma(r)\sqrt{Q}\big)^*X_{r-}|^2
    \,\d r
    \right)^{1/2}\\
    &\leq 2c_1pH_2(t)\E\left(
    \int_0^{t\wedge\tau_n}(1+|X_{r-}|^2)^{2p-1}
    \,\d r
    \right)^{1/2}\\
    &\leq 2c_1p\sqrt{t}H_2(t)
    \E\sup_{s\in[0,t\wedge\tau_n)}
    (1+|X_s|^2)^{p-1/2}.
\end{align*}
Applying the following inequality
(recall $p\geq1/2$)
\begin{equation}\label{young}
    yz^{p-1/2}\leq
    \frac{[3(2p-1)]^{2p-1}}{(2p)^{2p}}\,
    y^{2p}+\frac13\,z^p,\quad y,z\geq0,
\end{equation}
it holds that
$$
    \E\sup_{s\in[0,t\wedge\tau_n)}|I_2(s,p)|
    \leq[3(2p-1)]^{2p-1}
    [c_1\sqrt{t}H_2(t)]^{2p}
    +\frac13\,\E\sup_{s\in[0,t\wedge\tau_n)}
    (1+|X_s|^2)^p.
$$
Note that
\begin{align*}
    \E\sup_{s\in[0,t\wedge\tau_n)}|I_3(s,p)|
    &=\E\sup_{s\in[0,t\wedge\tau_n)}
    \left|
    \sum_{r\in[0,s],\,|\triangle Z_r|\geq1}
    J_1(\triangle Z_r,r,p)
    \right|\\
    &\leq\E\left[
    \sum_{r\in[0,t\wedge\tau_n),\,|\triangle Z_r|\geq1}
    |J_1(\triangle Z_r,r,p)|
    \right]\\
    &=\E\int_0^{t\wedge\tau_n}\int_{|x|\geq1}
    |J_1(x,r,p)|\,N(\d r,\d x)\\
    &=\E\int_0^{t\wedge\tau_n}\int_{|x|\geq1}
    |J_1(x,r,p)|\,\nu_Z(\d x)\d r.
\end{align*}
Since there exist $c_2=c_2(p)>0$ and nondecreasing
function $H_3:[0,\infty)\rightarrow(0,\infty)$ such that
$$
    |J_1(x,r,p)|\leq(1+|X_{r-}+\sigma(r)x|^2)^p+(1+|X_{r-}|^2)^p
    \leq c_2(1+|X_{r-}|^2)^p
    +c_2H_3(r)|x|^{2p},
$$
we know that
\begin{align*}
    \E\sup_{s\in[0,t\wedge\tau_n)}|I_3(s,p)|
    &\leq c_2\int_{|x|\geq1}\nu_Z(\d x)
    \cdot\E\int_0^{t\wedge\tau_n}
    (1+|X_{r-}|^2)^p\,\d r\\
    &\quad+c_2\int_0^{t\wedge\tau_n}H_3(r)\,\d r
    \cdot\int_{|x|\geq1}|x|^{2p}\,\nu_Z(\d x)\\
    &\leq c_2\int_{|x|\geq1}\nu_Z(\d x)
    \cdot\int_0^t
    \E\sup_{s\in[0,r\wedge\tau_n)}
    (1+|X_s|^2)^p\,\d r\\
    &\quad+c_2tH_3(t)
    \int_{|x|\geq1}|x|^{2p}\,\nu_Z(\d x).
\end{align*}
By the Burkholder-Davis-Gundy inequality,
cf. Novikov \cite[Theorem 1.1\, (a)]{Nov}, there
exists $c_3>0$ such that
$$
    \E\sup_{s\in[0,t\wedge\tau_n)}|I_4(s,p)|
    \leq c_3\E\left(
    \int_0^{t\wedge\tau_n}\int_{|x|<1}
    |J_1(x,r,p)|^2
    \,\nu_Z(\d x)\d r
    \right)^{1/2}.
$$
It is easy to verify that there exists a nondecreasing
function $H_4:[0,\infty)\rightarrow[1,\infty)$
such that
\begin{equation}\label{jh4fu}
    H_4(r)^{-1}\leq\frac{1+|X_{r-}+\sigma(r)x|^2}
    {1+|X_{r-}|^2}\leq H_4(r),\quad |x|<1,r\geq0.
\end{equation}
Combining this with the following elementary inequality
$$
    |y^p-z^p|\leq p(y^{p-1}+z^{p-1})|y-z|,\quad y,z\geq0,
$$
one has
\begin{align*}
    |J_1(x,r,p)|
    &\leq p\left[
    (1+|X_{r-}+\sigma(r)x|^2)^{p-1}
    +(1+|X_{r-}|^2)^{p-1}
    \right]\cdot\left|
    |X_{r-}+\sigma(r)x|^2-|X_{r-}|^2
    \right|\\
    &\leq p\left(
    H_4(r)^{|p-1|}+1
    \right)(1+|X_{r-}|^2)^{p-1}
    \cdot\left|
    |X_{r-}+\sigma(r)x|^2-|X_{r-}|^2
    \right|.
\end{align*}
For $|x|<1$, since it holds for some
nondecreasing function
$H_5:[0,\infty)\rightarrow(0,\infty)$ that
\begin{equation}\label{dg223d}
\begin{aligned}
    \left|
    |X_{r-}+\sigma(r)x|^2-|X_{r-}|^2
    \right|
    &=\left|2\<X_{r-},\sigma(r)x\>+|\sigma(r)x|^2
    \right|\\
    &\leq2|X_{r-}||\sigma(r)x|+|\sigma(r)x|^2\\
    &\leq H_5(r)(1+|X_{r-}|^2)^{1/2}|x|,
\end{aligned}
\end{equation}
we obtain
$$
    |J_1(x,r,p)|
    \leq
    p\left(
    H_4(r)^{|p-1|}+1
    \right)H_5(r)(1+|X_{r-}|^2)^{p-1/2}|x|.
$$
This yields that
\begin{align*}
    &\E\sup_{s\in[0,t\wedge\tau_n)}|I_4(s,p)|\\
    &\quad\leq c_3p\left(
    H_4(t)^{|p-1|}+1
    \right)H_5(t)\sqrt{t}\left(
    \int_{|x|<1} |x|^2\,\nu_Z(\d x)
    \right)^{1/2}
    \E\sup_{s\in[0,r\wedge\tau_n)}
    (1+|X_s|^2)^{p-1/2}\\
    &\quad\leq [3(2p-1)]^{2p-1}
    \left[
    2^{-1}c_3
    \left(
    H_4(t)^{|p-1|}+1
    \right)H_5(t)\sqrt{t}
    \left(
    \int_{|x|<1} |x|^2\,\nu_Z(\d x)
    \right)^{1/2}
    \right]^{2p}\\
    &\quad\quad+\frac13\,\E\sup_{s\in[0,t\wedge\tau_n)}
    (1+|X_s|^2)^p,
\end{align*}
where in the last inequality we have
used \eqref{young}. By the inequality
$$
    \left|y^p-z^p-pz^{p-1}(y-z)
    \right|
    \leq\frac{p|p-1|}{2}\,\left(
    y^{p-2}+z^{p-2}
    \right)(y-z)^2,\quad y,z\geq0,
$$
\eqref{jh4fu} and \eqref{dg223d}, we get that for $|x|<1$,
\begin{align*}
    &|J_2(x,r,p)|\\
    &\leq\frac{p|p-1|}{2}\left[
    (1+|X_{r-}+\sigma(r)x|^2)^{p-2}
    +(1+|X_{r-}|^2)^{p-2}
    \right]\left(
    |X_{r-}+\sigma(r)x|^2-|X_{r-}|^2
    \right)^2\\
    &\leq\frac{p|p-1|}{2}\left(
    H_4(r)^{|p-2|}+1
    \right)H_5(r)^2(1+|X_{r-}|^2)^{p-1}|x|^2.
\end{align*}
This implies that
\begin{align*}
    &\E\sup_{s\in[0,t\wedge\tau_n)}|I_5(s,p)|
    \leq\E\int_0^{t\wedge\tau_n}\int_{|x|<1}
    |J_2(x,r,p)|\,\nu_Z(\d x)\d r\\
    &\qquad\leq\frac{p|p-1|}{2}\left(
    H_4(t)^{|p-2|}+1
    \right)H_5(t)^2
    \int_{|x|<1}|x|^2\,\nu_Z(\d x)
    \cdot\int_0^t\E\sup_{s\in[0,r\wedge\tau_n)}
    (1+|X_s|^2)^p\,\d r.
\end{align*}
Substituting the above estimates into \eqref{esti},
we conclude that there exist $C=C(p)>0$ and
nondecreasing function
$\Phi:[0,\infty)\rightarrow(0,\infty)$ such that
$$
    \E\sup_{s\in[0,t\wedge\tau_n)}
    (1+|X_s|^2)^p\leq3\E(1+|X_0|^2)^p
    +\Phi(t)^C+\Phi(t)^{C}
    \int_0^t\E\sup_{s\in[0,r\wedge\tau_n)}
    (1+|X_s|^2)^p\,\d r.
$$
By Gronwall's inequality and letting $n\rightarrow\infty$,
we obtain that for all $t>0$
$$
    \E\sup_{s\in[0,t)}
    (1+|X_s|^2)^p
    \leq\left[
    3\E(1+|X_0|^2)^p
    +\Phi(t)^C
    \right]\cdot\exp\left[t\Phi(t)^{C}\right]
    <\infty,
$$
which completes the proof.
\end{proof}

\beg{thebibliography}{99}

\bibitem{App09}
    D. Applebaum,
    \emph{L\'{e}vy Processes and Stochastic Calculus}. Cambridge University Press, Cambridge 2009 (2nd ed).

\bibitem{BR1} V. Barbu, M. R\"ockner, \emph{Probabilistic representation for solutions to non-linear Fokker-Planck equations,} SIAM J. Math. Anal. 50(2018), 4246-4260.

\bibitem{BR2} V. Barbu, M. R\"ockner, \emph{From non-linear Fokker-Planck equations to solutions of distribution dependent SDE,} Ann. Probab. 48(2020), 1902-1920.


\bibitem{BLP15} M. Barczy, Z. Li, G. Pap, \emph{Yamada-Watanabe results for stochastic differential equations
    with jumps,} Int. J. Stoch. Anal. 2015,
    Art. ID 460472, 23pp.

\bibitem{BKRS} V. Bogachev, A. Krylov, M. R\"ockner, S. Shaposhnikov, \emph{Fokker-Planck-Kolmogorov Equations,}  Monograph, AMS(2015).




\bibitem{CA} K. Carrapatoso, \emph{Exponential convergence to equilibrium for the homogeneous Landau equation with hard potentials,} Bull. Sci. Math. 139(2015), 777-805.

\bibitem{DH} C.-S. Deng, X. Huang, \emph{Harnack inequalities for functional SDEs driven by subordinate Brownian motions,}
    Potential Anal. 56(2022), 213-226.

 \bibitem{DV1} L. Desvillettes, C. Villani, \emph{On the spatially homogeneous Landau equation for hard potentials, Part I :
existence, uniqueness and smothness,}  Comm. Part. Diff. Equat. 25(2000), 179-259.
\bibitem{DV2}  L. Desvillettes, C. Villani, \emph{On the spatially homogeneous Landau equation for hard potentials, Part II:
H-Theorem and Applications,} Comm. Part. Diff. Equat. 25(2000), 261-298.



\bibitem{FG} N. Fournier, A. Guillin, \emph{From a Kac-like particle system to the Landau equation for hard potentials and Maxwell molecules,}
 Ann. Sci. l'ENS 50(2017), 157-199.


\bibitem{Gu} H. Gu\'erin, \emph{Existence and regularity of a weak function-solution for some Landau equations with a stochastic
approach,} Stochastic Process Appl. 101(2002), 303-325.



\bibitem{HRW} X. Huang, M. R\"{o}ckner, F.-Y. Wang, \emph{
Non-linear Fokker--Planck equations for probability measures on path space and path-distribution dependent SDEs,}  Discrete Contin. Dyn. Syst. 39(2019), 3017-3035.

\bibitem{IW} I. Ikeda, S. Watanabe, \emph{Stochastic Differential Equations and Diffusion Processes,} North-Holland, Amsterdam 1981.
\bibitem{JMW} B. Jourdain, S. M\'{e}l\'{e}ard, W. A. Woyczynski, \emph{Nonlinear SDEs driven by L\'{e}vy processes
    and related PDEs,} ALEA Lat. Am. J. Probab. Math. Stat.
    4(2008), 1-29.


\bibitem{LMW} M. Liang, M. B. Majka, J. Wang, \emph{Exponential ergodicity for SDEs and
McKean-Vlasov processes with L\'{e}vy noise,} Ann. Inst. Henri Poincar\'e Probab. Stat. 57(2021), 1665-1701.


\bibitem{Nov} A.A. Novikov, \emph{On discontinuous
martingales,}  Theory Probab. Appl. 20(1975), 11-26.







\bibitem{RXZ20} M. R\"ockner, L. Xie, X. Zhang,
\emph{Superposition principle for non-local
Fokker-Planck-Kolmogorov operators,} Probab. Theory Related
    Fields 178(2020), 699-733.

\bibitem{Sato}
K. Sato, \emph{L\'{e}vy Processes and Infinitely Divisible
Distributions}. Cambridge University Press,
Cambridge 2013 (2nd ed).

\bibitem{SSV12}
    R.L. Schilling, R. Song, Z. Vondra\v{c}ek, \emph{Bernstein Functions. Theory and Applications}. De Gruyter, Studies in Mathematics vol.\ 37, Berlin 2012 (2nd ed).


\bibitem{Song}  Y. Song, \emph{Gradient estimates and exponential ergodicity for mean-field SDEs,} J. Theort. Probab. 33(2020), 201-238.



\bibitem{V2} C. Villani, \emph{On the spatially homogeneous Landau equation for Maxwellian Mocecules,} Math. Mod. Meth. Appl. Sci. 8(1998), 957-983.

\bibitem{Wan97} F.-Y. Wang, \emph{Logarithmic Sobolev inequalities on noncompact Riemannian manifolds,} Probab. Theory Related
    Fields 109(1997), 417-424.




\bibitem{Wbook}  F.-Y. Wang,  \emph{Harnack Inequalities and
Applications for Stochastic Partial Differential Equations}.
Springer, Berlin 2013.








\bibitem{W16} F.-Y. Wang, \emph{Distribution-dependent SDEs for Landau type equations,} Stochastic Process Appl. 128(2018), 595-621.




\bibitem{WW} F.-Y. Wang, J. Wang,  \emph{Harnack  inequalities for stochastic equations driven by L\'{e}vy noise,} J. Math. Anal. Appl. 410(2014), 513-523.


\bibitem{ZXC} X. Zhang, \emph{Derivative formulas and gradient estimates for SDEs driven by $\alpha$-stable processes,} Stochastic Process Appl. 123(2013), 1213-1228.
\end{thebibliography}

\end{document}